\documentclass[12pt]{amsart}
\newcommand{\bi}{\bar{i}}
\newcommand{\bj}{\bar{j}}
\newcommand{\bk}{\bar{k}}
\newcommand{\bl}{\bar{l}}

\newcommand{\bn}{\bar{n}}
\newcommand{\bp}{\bar{p}}
\newcommand{\bq}{\bar{q}}

\newcommand{\bz}{\bar{z}}

\newcommand{\bM}{\bar{M}}

\newcommand{\bpartial}{\bar{\partial}}

\newcommand{\fg}{\mathfrak{g}}

\newcommand{\fRe}{\mathfrak{Re}}

\newcommand{\tr}{\mbox{tr}}

\newcommand{\ol}{\overline}
\newcommand{\ul}{\underline}

\newtheorem{theorem}{Theorem}[section]
\newtheorem{lemma}[theorem]{Lemma}
\newtheorem{proposition}{Proposition}

\theoremstyle{definition}

\newtheorem{remark}{Remark}

\numberwithin{equation}{section}

\title[Complex Monge-Amp\`ere Type Equation]
{A Monge-Amp\`ere Type Fully Nonlinear Equation \\ on Hermitian Manifolds}

\author[Bo Guan and Qun Li]{}


 \email{guan@math.osu.edu}
 \email{qun.li@wright.edu}

\thanks{Part of research in this work was carried out while the authors
were supported in part by NSF grants}

\begin{document}
\maketitle

% Enter the first author's name and address:
\centerline{\scshape Bo Guan }
\medskip
{\footnotesize
% please put the address of the first author
 \centerline{Department of Mathematics}
   \centerline{The Ohio State University}
   \centerline{Columbus, OH 43210, USA}
} % Do not forget to end the {\footnotesize by the sign }

\medskip

\centerline{\scshape Qun Li}
\medskip
{\footnotesize
 % please put the address of the second  and third author
 \centerline{Department of Mathematics and Statistics}
   \centerline{Wright State University}
   \centerline{Dayton, OH 45435, USA}
}

%\bigskip
%
%% The name of the associate editor will be entered by an editorial staff
%% "Communicated by the associate editor name" is not needed for special issue.
% \centerline{(Communicated by the associate editor name)}

%The abstract of your paper
\begin{abstract}
We study a fully nonlinear equation of complex Monge-Amp\`ere type on
Hermitian manifolds. We establish the {\em a priori} estimates for
solutions of the equation up to the second order derivatives with the help
of a subsolution.

\end{abstract}

\bigskip

\section{Introduction}
\label{gblq-I}
\setcounter{equation}{0}
\medskip

Let $(M^n, \omega)$ be a compact Hermitian manifold of
dimension greater than  $1$, with smooth boundary $\partial M$ (which may be
empty),
 % with a $(1,1)$ form $\omega > 0$.
%Recall that $\omega$ is {\sl K\"ahler} if it is closed, i.e, $d \omega = 0$.
%with boundary $\partial M$. ($M$ is {\sl closed} when $\partial M$.)
% and $\bM = M \cup \partial M$.
and $\chi$ a smooth real $(1,1)$ form on $M$. Define
\[ \chi_u = \chi + \frac{\sqrt{-1}}{2} \partial \bpartial u
\;\; \mbox{and} \; [\chi] = \big\{\chi_u:  \, u \in C^2 (M)\big\}. \]
%A function $u \in C^2 (M)$ is called {\sl strictly $\chi$-plurisubhamonic}
%($\chi$-{\sl spsh}) or {\sl admissible} if
% and $\psi \in C^{\infty} (M \times \bfR)$. % $\psi > 0$.
%Let $\psi \in C^{\infty} (M)$, $\psi > 0$.
In this paper we are concerned with the equation
 \begin{equation}
\label{CH-I10}
 \chi_u^n =  \psi \chi_u^{n-1} \wedge \omega, \; \chi_u > 0
\;\; \mbox{on $M$}.
\end{equation}

When  $M$ is closed, both $\omega$, $\chi$ are K\"ahler and
$\psi$ is constant,
equation~\eqref{CH-I10} was introduced by Donaldson~\cite{Donaldson99a}
in the setting of moment maps.
%and by Chen~\cite{Chen00b} as limit of the $J$-flow.
Donaldson observed that in this case the solution of
equation~\eqref{CH-I10} is unique (up to a constant) if exists, and that a
necessary condition for the existence of solution is
$[n \chi - \psi \omega] > 0$.
He remarked that a natural conjecture would be that this is also a
sufficient condition.

Donaldson's problem was studied by Chen~\cite{Chen04},
Weinkove~\cite{Weinkove04}, \cite{Weinkove06}, Song and Weinkove~\cite{SW08}
using parabolic methods as limit of the $J$-flow introduced by
Donaldson~\cite{Donaldson99a} and Chen~\cite{Chen00b}.
In \cite{SW08} Song and Weinkove gave a necessary and sufficient condition
for convergence of the $J$-flow. Later on
Fang, Lai and Ma~\cite{FLM11} extended their approach and solved the equation
 \begin{equation}
\label{CH-I10alpha}
 \chi_u^n =  c_{\alpha} \chi_u^{n-\alpha} \wedge \omega^{\alpha},
  \; \chi_u > 0 \;\; \mbox{on $M$}.
\end{equation}
for all $1 \leq {\alpha} < n$.

A key ingredient in solving elliptic or parabolic fully nonlinear equations
is to derive {\em a priori} estimates up to the second order derivatives.
For the complex Monge-Amp\`ere equation on closed K\"ahler manifolds, these
estimates were established by Yau~\cite{Yau78} and Aubin~\cite{Aubin78}.
Their results and techniques had far-reaching influences in both
geometry and to the theory of nonlinear PDEs on manifolds.
In 1987, Cherrier~\cite{Cherrier87} studied the complex Monge-Amp\`ere
equation on Hermitian manifolds. He established the estimates for second
order derivatives in the general case, and extended Yau's zeroth
order estimate under an additional assumption on the Hermitian metric.
Recently, Tosatti and Weinkove~\cite{TWv10b}  were able
to carry out Yau's estimate on general closed Hermitian mannifolds.

There has been increasing interest to study fully nonlinear elliptic and
parabolic equations other than the complex Monge-Amp\`ere equation on
K\"ahler or Hermitian manifolds, both from geometric problems such as
Donaldson's problem mentioned above, and from the PDE point of view.
In this paper our main interest is to seek general technical methods in
establishing {\em a priori} estimates. We shall confine ourselves to
$\alpha = 1$ in \eqref{CH-I10alpha} but our method works for more
general equations and in particular for all $\alpha < n$. We shall treat
the other cases in separate papers.% (e.g. \cite{GL}, \cite{GS}).

Our first result for equation~\eqref{CH-I10} is the following

\begin{theorem}
\label{gl-thm10}
Let
%($\partial M = \emptyset$)
$u \in C^{4} (M)$ a solution of equation~\eqref{CH-I10}
%on a closed Hermitian manifold $(M^n, \omega)$.
and set $C_0 = \sup_M u - \inf_M u$.
 Assume that there exists a
function $\ul{u} \in C^2 (M)$ satisfying
 \begin{equation}
\label{CH-I20}
\chi_{\ul{u}}^n \geq  \psi \chi_{\ul{u}}^{n-1} \wedge \omega, \;
\chi_{\ul{u}} > 0
\;\; \mbox{on $M$}.
\end{equation}
Then there are constants
 $C_1$, $C_2$, depending on $C_0$,
$|\ul{u}|_{C^2{(M})}$, the positive lower bound of $\chi_{\ul{u}}$, and
$\inf_M \psi > 0$ as well as other known data, such that
 \begin{equation}
\label{CH-I30}
%\sup_M u - \inf_M u \leq C_0, \;\;
\max_M |\nabla u| \leq C_1 (1 + \max_{\partial M} |\nabla u|), \;\;
|\Delta u| \leq C_2 ((1 + \max_{\partial M} |\Delta u|)
\;\; \mbox{on $M$}.
\end{equation}
\end{theorem}

We remark that  both $C_1$ and $C_2$ in Theorem~\ref{gl-thm10} depend
on $C_0$, but the estimate for $\Delta u$ is independent of the gradient
bound. (i.e. $C_2$ is independent of $C_1$.)
Apparently, assumption~\eqref{CH-I20}  is a trivial necessary
condition for the solvability of equation~\eqref{CH-I10}.
Following the literature we shall call the function $\ul{u}$
a {\em subsolution} of equation~\eqref{CH-I10}.
It seems worthwhile to remark that the subsolution $\ul{u}$
plays key roles in our proof of both estimates in \eqref{CH-I30};
% which is the centerpieces of the proof of Theorem~\ref{gl-thm10};
%we need to make use of the subsolution $\ul{u}$ in an crucial way;
see Sections~\ref{glq-C1}-\ref{glq-C2} for details. This appears to us a
rather new phenomena, and we are not clear how to derive
these estimates without using $\ul{u}$.
We also remark that the gradient estimate seems new even in the K\"ahler
case.

Theorem~\ref{gl-thm10} still holds under the following
assumption which is slightly weaker than \eqref{CH-I20}
\begin{equation}
\label{CH-I20'}
(n \chi_{\ul{u}} - (n-1) \psi \omega) \wedge  \chi_{\ul{u}}^{n-2} > 0,  \;
\chi_{\ul{u}} > 0
\;\; \mbox{on $M$}.
\end{equation}
When $M$ is K\"ahler and $\psi$ is a constant, this condition was first
given by Song and Weinkeve~\cite{SW08} and proved to be necessary and
sufficient for the solvability ~\eqref{CH-I10} on closed  K\"ahler manifolds.

The estimates in Theorem~\ref{gl-thm10} enable us to treat
the Dirichlet problem for equation~\eqref{CH-I10} on Hermitian manifolds
with boundary.
More precisely we can prove the following existence result under
the assumption of existence of a subsolution.

\begin{theorem}
\label{gl-thm20}
Let $(M^n, \omega)$ be a compact Hermitian manifold with smooth boundary
$\partial M$, $\psi \in C^{\infty} (\bM)$, $\psi > 0$,
where $\bM = M \cup \partial M$ and
$\varphi \in C^{\infty} (\partial M)$.
Suppose there exists a subsolution $\ul{u} \in C^2 (\bM)$ satisfying
 \begin{equation}
\label{CH-I20b}
 \left\{ \begin{aligned}
\chi_{\ul{u}}^n  \,& \geq \psi \chi_{\ul{u}}^{n-1} \wedge \omega, \;
\chi_{\ul{u}}  > 0  \;\; \mbox{on $\bM$}  \\
  \,& \ul{u} = \varphi \;\; \mbox{on $\partial M$}.
  \end{aligned}  \right.
\end{equation}
Then equation~\eqref{CH-I10} admits a unique
solution $u \in C^{\infty} (M)$ with $u = \varphi$ on $\partial M$.
\end{theorem}

In order to prove Theorem~\ref{gl-thm20} we need to establish
{\em a priori} boundary estimates. The gradient estimate on the boundary
follows immediately from a barrier argument. The proof for the second
order boundary estimates is similar to the Monge-Amp\`ere equation
case in \cite{GL10} and will be omitted here. We shall come back to
the issue for more general equations including (\ref{CH-I10alpha}) in our
forthcoming papers
where we shall also discuss the existence questions for the closed
manifold case. In this paper we will just present the global {\em a priori}
estimates up to the second order derivatives of the solutions to
equation~(\ref{CH-I10}).

The rest of this paper is organized as follows. In section~\ref{glq-P}
we fix some notations and introduce some fundamental formulas in Hermitian
geometry, which will be used throughout the paper. We also establish a
crucial lemma in this section that will be applied to deriving the estimates
in the following sections. Section~\ref{glq-C1} and Section~\ref{glq-C2} will
 be devoted to establishing the global gradient estimates and the estimates
for the second order derivatives respectively. In both sections we make the
important use of the existence of a {\em subsolution}.

We dedicate this article with sincere respect and admiration to Professor
Avner Friedman on the occasion of his 80th birthday. We wish to thank
Wei Sun for pointing out several mistakes in previous versions.
Part of this work was done while the first author was in Xiamen
University in summer 2011.

%\bigskip

\section{Preliminaries}
\label{glq-P}
\setcounter{equation}{0}
\medskip

We shall follow the notations in \cite{GL10} where the reader can also
find a brief introduction to the background materials for Hermitian
manifolds.
In particular, $g$ and $\nabla$ will denote the
Riemannian metric and Chern connection of $(M, \omega)$. The
torsion and curvature tensors of $\nabla$ are defined by
\begin{equation}
\label{cma-K95}
\begin{aligned}
   T (u, v)  = \,& \nabla_u v - \nabla_v u - [u,v], \\
 R (u, v) w  = \,& \nabla_u \nabla_v w - \nabla_v \nabla_u w - \nabla_{[u,v]} w,
\end{aligned}
\end{equation}
respectively.
In local coordinates $z = (z_1, \ldots, z_n)$, %$g$, $T$ and $R$ are given by
\begin{equation}
\label{cma-K70}
\left\{ \begin{aligned}
 g_{i \bj} \,& = g \Big(\frac{\partial}{\partial z_i},
 \frac{\partial}{\partial \bz_j}\Big), \;\; \{g^{i\bj}\} = \{g_{i\bj}\}^{-1}, \\
 %\Gamma^k_{ij} \partial_k \\
  T^k_{ij} \,& = \Gamma^k_{ij} - \Gamma^k_{ji}
   = g^{k\bl} \Big(\frac{\partial g_{j\bl}}{\partial z_i}
           - \frac{\partial g_{i\bl}}{\partial z_j}\Big),  \\
 R_{i\bj k\bl} \,&   =  - g_{m \bl} \frac{\partial \Gamma_{ik}^m}{\partial \bz_j}
 = - \frac{\partial^2 g_{k\bl}}{\partial z_i \partial \bz_j}
       + g^{p\bq} \frac{\partial g_{k\bq}}{\partial z_i}
                  \frac{\partial g_{p\bl}}{\partial \bz_j}.
\end{aligned} \right.
\end{equation}

Let $v \in C^4 (M)$. For convenience we write in local coordinates
\[ v_{i\bj} = \nabla_{\bj} \nabla_i v,  v_{i\bj k} = \nabla_k v_{i\bj},
   \; \mbox{etc.} \]
Recall $v_{i\bj} = v_{\bj i} = \partial_i \bpartial_j v$ and
$v_{i\bj k} = \partial_k v_{i\bj} - \Gamma_{ki}^l v_{l\bj}$. It follows that
\begin{equation}
\label{gblq-B145}
\left\{
\begin{aligned}
 v_{i \bj k} - v_{k \bj i} = \,& T_{ik}^l v_{l\bj},  \\
v_{i \bj \bk} - v_{i \bk \bj} = \,& \ol{T_{jk}^l} v_{i\bl}.
\end{aligned}
 \right.
\end{equation}
We calculate
\begin{equation}
\label{gblq-B146}
 \begin{aligned}
v_{i\bj k\bl}
   = \,&  \bpartial_l v_{i\bj k} - \ol{\Gamma_{lj}^q} v_{i\bq k} \\
   = \,&  \bpartial_l (\partial_k v_{i\bj} - \Gamma_{ki}^p v_{p\bj})
       - \ol{\Gamma_{lj}^q} v_{i\bq k} \\
   = \,&  \bpartial_l \partial_k v_{i\bj} - \bpartial_l \Gamma_{ki}^p v_{p\bj}
       - \Gamma_{ki}^p (v_{p\bj \bl} + \ol{\Gamma_{lj}^q} v_{p\bq})
       - \ol{\Gamma_{lj}^q} v_{i\bq k} \\
   = \,&  \partial_k \bpartial_l v_{i\bj} + g^{p\bq} R_{k\bl i\bq} v_{p\bj}
       - \Gamma_{ki}^p v_{p\bj \bl} - \ol{\Gamma_{lj}^q} v_{i\bq k}
       - \Gamma_{ki}^p \ol{\Gamma_{lj}^q} v_{p\bq},
\end{aligned}
\end{equation}
\[ v_{i \bj \bl k} = \ol{v_{j\bi l \bk}}
  =  \bpartial_l \partial_k v_{i\bj}
       - \ol{\Gamma_{lj}^p} \Gamma_{ki}^q v_{q\bp}
       - \ol{\Gamma_{lj}^p} v_{i\bp k} - \Gamma_{ki}^q v_{\bj q \bl}
       + g^{q\bp} R_{k\bl \bj q} v_{i\bp}. \]
Therefore,
\begin{equation}
\label{gblq-B147}
\left\{ \begin{aligned}
v_{i\bj k\bl} - v_{i\bj \bl k}
      = \,& g^{p\bq} R_{k\bl i\bq} v_{p\bj}
          - g^{p\bq} R_{p \bl k \bj} v_{i\bq}, \\
v_{i \bj k \bl} - v_{k \bl i \bj}
  = \,&  g^{p\bq} (R_{k\bl i\bq} v_{p\bj} - R_{i\bj k\bq} v_{p\bl})
        + T_{ik}^p v_{p\bj \bl} + \ol{T_{jl}^q} v_{i\bq k}
        - T_{ik}^p \ol{T_{jl}^q} v_{p\bq}.
 %       + (\Gamma_{ik}^p \ol{\Gamma_{jl}^q} - \Gamma_{ki}^p \ol{\Gamma_{lj}^q}) v_{p\bq}.
 \end{aligned}  \right.
\end{equation}
The second identity in \eqref{gblq-B147} follows from \eqref{gblq-B146},
\eqref{gblq-B145} and
\[ \Gamma_{ki}^p \ol{\Gamma_{lj}^q} - \Gamma_{ik}^p \ol{\Gamma_{jl}^q}
   + T_{ik}^p \ol{\Gamma_{jl}^q} + \Gamma_{ik}^p \ol{T_{jl}^q} = T_{ik}^p \ol{T_{jl}^q}. \]
It can also be derived as follows.
\begin{equation}
\label{gblq-R150}
 \begin{aligned}
v_{i \bj k \bl} - v_{k \bl i \bj}
 = \,& (v_{i \bj k \bl} - v_{k \bj i \bl})
         + (v_{k \bj i \bl} - v_{k \bj \bl i}) \\
     &  + (v_{k \bj \bl i} - v_{k \bl \bj i})
         + (v_{k \bl \bj i}- v_{k  \bl i \bj}) \\
 = \,& \nabla_{\bl} (T_{ik}^p v_{p\bj})
         +  g^{p\bq} R_{i\bl k\bq} v_{p\bj}
          - g^{p\bq} R_{i\bl p \bj} v_{k\bq} \\
     & + \nabla_i (\ol{\Gamma_{jl}^q} v_{k\bq})
         - g^{p\bq} R_{i\bj k\bq} v_{p\bl}
          + g^{p\bq} R_{i\bj p \bl} v_{k\bq} \\
 = \,& \nabla_{\bl} T_{ik}^p v_{p\bj} + T_{ik}^p v_{p\bj \bl}
       +  g^{p\bq} R_{i\bl k\bq} v_{p\bj}
       - g^{p\bq} R_{i\bl p \bj} v_{k\bq} \\
     & + \nabla_i \ol{\Gamma_{jl}^q} v_{k\bq} + \ol{T_{jl}^q} v_{k\bq i}
       - g^{p\bq} R_{i\bj k\bq} v_{p\bl}
       + g^{p\bq} R_{i\bj p \bl} v_{k\bq} \\
 = \,& g^{p\bq} (R_{k\bl i\bq} v_{p\bj} - R_{i\bj k\bq} v_{p\bl})
        + T_{ik}^p v_{p\bj \bl} + \ol{T_{jl}^q} v_{k\bq i} \\
  = \,& g^{p\bq} (R_{k\bl i\bq} v_{p\bj} - R_{i\bj k\bq} v_{p\bl})
        + T_{ik}^p v_{p\bj \bl} + \ol{T_{jl}^q} v_{i\bq k}
        - T_{ik}^p \ol{T_{jl}^q} v_{p\bq}.
 \end{aligned}
\end{equation}

%Furthermore, if $v_{i\bj}$ is diagonal,
%\begin{equation}
%\label{gblq-R155}
% \begin{aligned}
%v_{i \bi k \bk} - v_{k \bk i \bi}
%  = \,& R_{k\bk i\bi} v_{i\bi} - R_{i\bi k\bk} v_{k\bk}
% + 2 \fRe\{\ol{T_{ik}^q} v_{i\bq k}\} - T_{ik}^q \ol{T_{ik}^q} v_{q\bq}.
%  \end{aligned}
%\end{equation}

Let $u \in C^4 (M)$ be a solution of equation~\eqref{CH-I10}.
As in \cite{GL10},
we denote $\fg_{i\bj} = \chi_{i\bj} + u_{i\bj}$,
$\{\fg^{i\bj}\} =  \{\fg_{i\bj}\}^{-1}$
and let $W = \tr \chi + \Delta u$.
%Let $\mathcal{L}$ be the linear operator which is locally given by
%\[ \mathcal{L} v := F^{i\bj} v_{i\bj} \equiv \fg^{i\bj} v_{i\bj}
%      - \frac{\Delta v}{w}. \]
Assume that $g_{i\bj} = \delta_{ij}$
and $\fg_{i\bj}$ is diagonal at a fixed point $p \in M$. Then
\begin{equation}
\label{gblq-R155a}
 \begin{aligned}
u_{i \bi k \bk} - u_{k \bk i \bi}
   = \,& R_{k\bk i\bp} u_{p\bi} - R_{i\bi k\bp} u_{p\bk}
  + 2 \fRe\{\ol{T_{ik}^j} u_{i\bj k}\} -  T_{ik}^p \ol{T_{ik}^q} u_{p\bq},
  \end{aligned}
\end{equation}
and therefore,
\begin{equation}
\label{gblq-R155}
 \begin{aligned}
\fg_{i \bi k \bk} - \fg_{k \bk i \bi}
   = \,& R_{k\bk i\bi} \fg_{i\bi} - R_{i\bi k\bk} \fg_{k\bk}
         + 2 \fRe\{\ol{T_{ik}^j} \fg_{i\bj k}\}
         - |T_{ik}^j|^2 \fg_{j\bj} - G_{i\bi k\bk}
  \end{aligned}
 \end{equation}
where
 \begin{equation}
\begin{aligned}
  G_{i\bi k\bk}
   = \,& \chi_{k \bk i \bi} - \chi_{i \bi k \bk}
     +  R_{k\bk i\bp} \chi_{p\bi} - R_{i\bi k\bp} \chi_{p\bk}
     + 2 \fRe\{\ol{T_{ik}^j} \chi_{i\bj k}\}
     - T_{ik}^p \ol{T_{ik}^q} \chi_{p\bq}.
       \end{aligned}
\end{equation}

In local coordinates, equation~\eqref{CH-I10} can be written in the form
\begin{equation}
\label{CH-I10k=1}
\fg^{i\bj} g_{i\bj} =  \frac{n}{\psi}.
%\sum \frac{1}{\lambda_i (\chi_u)}  =  \frac{1}{\psi}
\end{equation}
%where $\lambda_1 (\chi_u), \ldots, \lambda_n (\chi_u)$ are the eigenvalues
%of $\chi_u$, i.e., in local coordinates, the roots of
%\[ \det (\fg_{i\bj} - \lambda g_{i\bj}) = 0. \]
Differentiating this equation twice gives at $p$
\begin{equation}
\label{gblq-M70}
\begin{aligned}
F^{i\bi} (u_{i\bi k} u_{\bk} + u_k u_{i\bi \bk})
    = & \,  2 \fRe \{f_k u_{\bk} - F^{i\bi} \chi_{i\bi k} u_{\bk}\}.
    %   \geq 2 f_u |\nabla u|^2 - C |\nabla u|,
       \end{aligned}
\end{equation}
\begin{equation}
\label{gblq-C75.1}
F^{i\bi} \fg_{i\bi k\bk}
 - (F^{i\bi} \fg^{j\bj} + F^{j\bj} \fg^{i\bi}) \fg_{i\bj k} \fg_{j\bi \bk}
   = f_{k\bk}
\end{equation}
where $F^{i\bi} = (\fg^{i\bi})^2$ and $f = - n \psi^{-1}$.
Note that $\sum F^{i\bi} \leq (\sum \fg^{i\bi})^2 \leq C$;
we shall use this fact without further reference.
By \eqref{gblq-R155} and \eqref{gblq-C75.1} we have
\begin{equation}
\label{gblq-C120.1}
\begin{aligned}
F^{i\bi} W _{i\bi}
 = \,& F^{i\bi} \fg_{j\bj i\bi}
 \geq \sum_k  F^{i\bi} \fg^{j\bj}
           (|\fg_{i\bj k} - T_{ik}^j \fg_{j\bj}|^2 + |\fg_{i\bj k}|^2) - C W.
    %    + n \Delta (\psi^{-1}) - C w.
        \end{aligned}
\end{equation}

Finally, we note that since $\ul{u} \in C^2 (M)$ and
$\chi_{\ul{u}} > 0$ on $M$,
\begin{equation}
\label{gblq-C214}
 \epsilon \omega \leq \chi_{\ul{u}} \leq \epsilon^{-1} \omega
 \end{equation}
for some $\epsilon > 0$. Consequently,
\begin{equation}
\label{gblq-C215}
\sum \fg^{i\bi} (\chi_{i\bi} + \ul{u}_{i\bi}) \geq \epsilon \sum \fg^{i\bi}.
\end{equation}

Let $\lambda_1 (\chi_{\ul{u}}), \ldots, \lambda_n (\chi_{\ul{u}})$
denote the eigenvalues of $\{\chi_{i\bj} + \ul{u}_{i\bj}\}$.
Then \eqref{CH-I20} is equivalent to
\begin{equation}
\label{CH-I10k=1'}
\sum \frac{1}{\lambda_i (\chi_{\ul{u}})}  \leq  \frac{n}{\psi}
\end{equation}
while \eqref{CH-I20'} is equivalent to
\begin{equation}
\label{CH-I10k=1a}
\sum_{i \neq k} \frac{1}{\lambda_i (\chi_{\ul{u}})} <  \frac{n}{\psi}
\,\, \mbox{for each $k = 1, \ldots, n$}.
\end{equation}

It is clear that at a point
where $g_{i\bj} = \delta_{ij}$ in local coordinates,
\begin{equation}
\label{CH-I10k=1''}
\sum \frac{1}{\chi_{i\bi} + \ul{u}_{i\bi}}  \leq
   \sum \frac{1}{\lambda_i (\chi_{\ul{u}})}  \leq  \frac{n}{\psi}.
\end{equation}

We conclude this section with
the following inequality which will play a crucial role in both the gradient
and second order estimates in the sections below.

\begin{lemma}
\label{gblq-lemma-C20}
%Let $u, \ul{u} \in C^2 (M)$ be a solution and subsolution of \eqref{cma2-M10},
%respectively, $\chi_u > 0$ and
%$\epsilon \omega \leq \chi_{\ul{u}} \leq \epsilon^{-1} \omega$.
There exist $\theta > 0$ and $N \geq n$ depending on $\epsilon$
such that if $W \geq N$ then
\begin{equation}
\label{gblq-C210}
F^{i\bj} (\chi_{i\bj} + \ul{u}_{i\bj}) \geq \frac{n + \theta}{\psi}.
\end{equation}
\end{lemma}

\begin{proof}
We may assume $g_{i\bj} = \delta_{ij}$ and
$\{\fg_{i\bj}\}$ is diagonal.
Suppose that $\fg_{1\bar{1}} \geq \dots \geq \fg_{n\bn}$.
By Schwarz inequality, \eqref{CH-I10k=1},  \eqref{CH-I10k=1''}
 and \eqref{gblq-C214} we have
\begin{equation}
\label{gblq-C95}
 \begin{aligned}
 \sum_{i\geq 2} (\fg^{i\bi})^2 (\chi_{i\bi} + \ul{u}_{i\bi})
  \geq \,& \Big(\sum_{i\geq 2} \fg^{i\bi}\Big)^2
           \Big/ \sum_{i\geq 2} \frac{1}{\chi_{i\bi} + \ul{u}_{i\bi}} \\
  \geq \,& \Big(\frac{n}{\psi} - \fg^{1\bar{1}}\Big)^2
           \Big/ \Big(\frac{n}{\psi} - \frac{1}
            {\chi_{1\bar{1}} + \ul{u}_{1\bar{1}}}\Big) \\
  \geq \,& \frac{(n- \psi \fg^{1\bar{1}})^2}{n \psi}
           \Big(1 + \frac{\psi}{n (\chi_{1\bar{1}} + \ul{u}_{1\bar{1}})}\Big)  \\
  \geq \,& \frac{(n- \psi \fg^{1\bar{1}})^2}{n \psi}
            \Big(1 + \frac{\epsilon \psi}{n}\Big) \\
     \geq \,& \frac{n + \theta}{\psi}
    \end{aligned}
  \end{equation}
provided $\fg_{1\bar{1}}$ is sufficiently large.
\end{proof}

\begin{remark}
One can replace assumption  \eqref{CH-I20} by  \eqref{CH-I20'} in
Lemma~\ref{gblq-lemma-C20}.
This is clear from the proof.
\end{remark}

\section{Gradient estimates}
\label{glq-C1}
\setcounter{equation}{0}
\medskip

Let $u \in C^{4} (M)$ be a solution of equation~\eqref{CH-I10}.
The primary goal of this section is to establish the {\em a priori} gradient
estimates.

\begin{proposition}
\label{gblq-prop-G10}
There exists a uniform constant $C > 0$
such that
\begin{equation}
\label{cma-40}
\max_{\bM} |\nabla u| \leq %\exp \{(u-L) \mbox{sign} (\mu)\}
 C (1 + \max_{\partial M}  |\nabla u|).  %\exp \{(L-u) \mbox{sign} (\mu)\}
\end{equation}
\end{proposition}

\begin{proof}
Let $\phi =A e^{\eta}$ where $\eta = \ul{u} - u$
and $A$ is a positive constant to be determined later.
Suppose the function $e^{\phi} |\nabla u|^2$ attains its maximum at
an interior point $p \in M$. We choose local coordinate around $p$ such that
 $g_{i\bj} = \delta_{ij}$ and $\fg_{i\bj}$ is diagonal at $p$ where, unless
otherwise indicated, the computations below are evaluated.

 For each $i = 1, \ldots, n$, we have
\begin{equation}
\label{gblq-G30}
  \frac{(|\nabla u|^2)_i}{|\nabla u|^2} + \phi_i = 0, \;\;
\frac{(|\nabla u|^2)_{\bi}}{|\nabla u|^2} + \phi_{\bi} = 0
\end{equation}
and
\begin{equation}
\label{gblq-G40}
 \frac{(|\nabla u|^2)_{i\bi}}{|\nabla u|^2}
-  \frac{|(|\nabla u|^2)_{i}|^2}{|\nabla u|^4}
  + \phi_{i\bi} \leq 0.
\end{equation}

A straightforward calculation shows that
\begin{equation}
\label{cma2-M50}
 (|\nabla u|^2)_i =  u_k u_{i\bk} + u_{ki} u_{\bk},
\end{equation}
%and, by \eqref{gblq-B150},
\begin{equation}
\label{cma2-M60}
\begin{aligned}
(|\nabla u|^2)_{i\bi}
= \, & u_{k \bi} u_{i \bk} + u_{ki} u_{\bk \bi}
         + u_{ki\bi} u_{\bk} + u_k u_{i\bk \bi} \\
 = \, & u_{ki} u_{\bk \bi}+ u_{i\bi k} u_{\bk} + u_{i \bi \bk} u_k
        + R_{i \bi k \bl} u_l u_{\bk} \\
      & +  \sum_k |u_{i\bk} - T^k_{il} u_{\bl}|^2- \sum_k |T^k_{il}  u_{\bl}|^2.
\end{aligned}
\end{equation}
 It follows that
\begin{equation}
\label{bglq-M80}
F^{i\bi} (|\nabla u|^2)_{i\bi}
  \geq F^{i\bi} u_{ki} u_{\bk \bi}
 %    + 2 \fRe \{(f)_k u_{\bk} - \fg^{i\bi} (\chi_{i\bi})_k u_{\bk}\} \\
   +   \sum_k F^{i\bi}|u_{i\bk} - T^k_{il} u_{\bl}|^2
   - C (1+ |\nabla u|^2).
\end{equation}

By \eqref{gblq-G30} and \eqref{cma2-M50},
\begin{equation}
\label{gblq-G50}
 |(|\nabla u|^2)_i|^2 = |u_{\bk} u_{ki}|^2
 - 2 |\nabla u|^2  \fRe \{u_k u_{i\bk} \phi_{\bi}\} - |u_k u_{i\bk}|^2.
\end{equation}
Combining \eqref{gblq-G40}, \eqref{gblq-G50} and \eqref{bglq-M80}, we obtain
\begin{equation}
\label{cma2-M40}
\begin{aligned}
 |\nabla u|^2 F^{i\bi} \phi_{i\bi} +  2 F^{i\bi} \fRe \{u_k u_{i\bk} \phi_{\bi}\}
    \leq C  (1 + |\nabla u|^2).
  \end{aligned}
\end{equation}

Now,
\[ \phi_i = \phi \eta_i, \;\; \phi_{i\bi}
    = \phi (\eta_i \eta_{\bi} + \eta_{i\bi}). \]
We have
\begin{equation}
\label{gblq-G60'}
\begin{aligned}
 2 \phi^{-1} F^{i\bi} \fRe\{u_k u_{i\bk} \phi_{\bi}\}
     = \,& 2 \fg^{i\bi} \fRe\{u_{i} \eta_{\bi}\}
        - 2 F^{i\bi} \fRe\{\chi_{i\bk} u_k  \phi_{\bi}\} \\
  \geq \,& - \frac{1}{2} |\nabla u|^2 F^{i\bi} \eta_{i} \eta_{\bi} - C
     \end{aligned}
\end{equation}
and
\begin{equation}
\label{gblq-G70}
\begin{aligned}
 \phi^{-1} F^{i\bi} \phi_{i\bi}
  = \, &  F^{i\bi} \eta_{i} \eta_{\bi} + F^{i\bi} \eta_{i\bi} \\
\geq \, &   F^{i\bi} \eta_{i} \eta_{\bi}
           + F^{i\bi} (\ul{u}_{i\bi} + \chi_{i\bi})
           - \frac{n}{\psi}. % \sum \fg^{i\bi}.
      \end{aligned}
\end{equation}
Therefore, by \eqref{cma2-M40},
\begin{equation}
\label{gblq-G70'}
\begin{aligned}
\frac{1}{2} F^{i\bi} \eta_{i} \eta_{\bi}
 + F^{i\bi} (\ul{u}_{i\bi} + \chi_{i\bi})
  - \frac{n}{\psi}
 \leq \, & C (\phi^{-1} + |\nabla u|^{-2}). % \sum \fg^{i\bi}.
      \end{aligned}
\end{equation}

We consider two cases separately: (a) $W > N$ for some $N$ sufficiently large,
and (b) $W \leq N$.

In case (a) we have
\[ F^{i\bi} (\ul{u}_{i\bi} + \chi_{i\bi})
  - \frac{n}{\psi} \geq \frac{\theta n}{\psi}, \]
by Lemma~\ref{gblq-lemma-C20}
for some $\theta > 0$. Therefore from \eqref{gblq-G70'}
we obtain a bound $|\nabla u| \leq C$ when $A$ is chosen
sufficiently large.

Suppose now that $W \leq N$. Then, using equation~\eqref{CH-I10k=1},
 \begin{equation}
%\label{gblq-G70}
\begin{aligned}
 F^{i\bi} \eta_{i} \eta_{\bi} + F^{i\bi} (\ul{u}_{i\bi} + \chi_{i\bi})
 \geq \, &  |\nabla \eta|^2 \min_{i} F^{i\bi}
         + \epsilon \sum F^{i\bi} \\
 \geq & \, n |\nabla \eta|^{\frac{2}{n}} \epsilon^{\frac{n-1}{n}}
          (\det \fg^{i\bj})^{\frac{2}{n}} \\
% = & \, \epsilon^{\frac{n-1}{n}} |B|^{\frac{1}{n}} |\nabla \eta|^{\frac{2}{n}}
%        \Big(\frac{\psi}{\sigma_{n-1}}\Big)^{\frac{2}{n}} \\
 \geq & \,  \epsilon^{\frac{n-1}{n}} |\nabla \eta|^{\frac{2}{n}}
        \Big(\frac{\psi}{W^{n-1}}\Big)^{\frac{2}{n}}  \\
 \geq & \, c_0 |\nabla \eta|^{\frac{2}{n}}.
     \end{aligned}
\end{equation}
Plugging this back in \eqref{gblq-G70'} %and choose $B$ sufficiently large,
we derive a bound $|\nabla \eta| \leq C$ which in turn implies a bound
 $|\nabla u| \leq C$.
\end{proof}

\bigskip

\section{The second order estimates}
\label{glq-C2}
\setcounter{equation}{0}
\medskip

In this section we derive second order estimates for solutions of
equation~\eqref{CH-I10}.

\begin{proposition}
\label{gblq-prop-C10}
Let $u \in C^4 (M)$ be a solution of equation~\eqref{CH-I10}.
There exists a constant $C > 0$ depending on $\epsilon$, $\sup \psi^{-1}$,
$\sup_M u - \inf_M u$,
the $C^2$ norms of $\chi$, $\ul{u}$ and $\psi$,
%\mbox{$ |u|_{C^1 (\bM)}$, %\;\; \inf_{M} (D_u \psi^{\frac{1}{n}})_u,
%$|\psi^{\frac{1}{n-1}}|_{C^2 (M)}$,
%$|\chi|_{C^2 (M)}$,  $|\ul{u}|_{C^2 (M)}$
and the geometric quantities of $M$,
%  \inf_{j, k} R_{j\bj k\bk}, \; \sup_k |R_{k\bk} - S_{k\bk}|, \; |T|^2 \]
such that
\[ \max_{\bM} |\Delta u| \leq  C (1 + \max_{\partial M}  |\Delta u|). \]

%\begin{equation} \label{gblq-I60}
%  w \leq C e^{A (u - \inf_M u)}.
%\end{equation}
\end{proposition}

\begin{proof}
Consider $\varPhi = e^\phi W$, where as aforementioned $W = \tr \chi + \Delta u$ and $\phi$
 %= A (\ul{u} - u)$ and $A$ is a positive constant
is a function to be determined.
Suppose $\varPhi$ achieves its maximum at a point $p \in M$.
Choose local coordinates such that $g_{i\bj} = \delta_{ij}$
and $\fg_{i\bj}$ is diagonal at $p$.
 We have (all calculations  below are done at $p$)
\begin{equation}
\label{gblq-C80}
  \frac{W_i}{W} + \phi_i = 0, \;\;
\frac{W_{\bi}}{W} + \phi_{\bi} = 0,
\end{equation}
\begin{equation}
\label{gblq-C90}
 \frac{W_{i\bi}}{W}
-  \frac{|W_{i}|^2}{W^2}
  + \phi_{i\bi} \leq 0.
\end{equation}
By \eqref{gblq-B145} and  \eqref{gblq-C80},
\begin{equation}
\label{gblq-C905a}
  \begin{aligned}
 |W_i|^2 = \,& \Big|\sum_j \fg_{j\bj i}\Big|^2
        = \Big|\sum_j  (\fg_{i\bj j} - T_{ij}^j \fg_{j\bj}) + \lambda_i\Big|^2 \\
      \leq \,& \Big|\sum_j (\fg_{i\bj j} - T_{ij}^j \fg_{j\bj})\Big|^2
     + 2 \sum_j (\fRe\{(\fg_{i\bj j} - T_{ij}^j \fg_{j\bj}) \ol{\lambda_i}\})
               + |\lambda_i|^2 \\
        = \,& \Big|\sum_j (\fg_{i\bj j} - T_{ij}^j \fg_{j\bj})\Big|^2
          - 2 W \fRe\{\phi_i  \ol{\lambda_i}\} - 2 |\lambda_i|^2
         \end{aligned}
\end{equation}
where
\[ \lambda_{i} = \sum_j (\chi_{j\bj i} - \chi_{i\bj j} + T_{ij}^l \chi_{l\bj}). \]
By Schwarz inequality,
\begin{equation}
\label{gblq-C905b}
%  \begin{aligned}
\Big|\sum_j (\fg_{i\bj j} - T_{ij}^j \fg_{j\bj})\Big|^2
\leq W \sum_j \fg^{j\bj} |\fg_{i\bj j} - T_{ij}^j \fg_{j\bj}|^2.
%         \end{aligned}
\end{equation}
It therefore follows from \eqref{gblq-C90}, \eqref{gblq-C905a},
\eqref{gblq-C905b} and \eqref{gblq-C120.1} that
\begin{equation}
\label{gblq-C94}
\begin{aligned}
0 \geq \,& F^{i\bi} W_{i\bi} - F^{i\bi} \frac{|W_i|^2}{W}
           + W F^{i\bi} \phi_{i\bi}
   \geq W F^{i\bi} \phi_{i\bi} + 2 F^{i\bi} \fRe\{\phi_i  \ol{\lambda_i}\} - C W.
%      = \,& w (\fg^{i\bi})^2 (\chi_{i\bi} + \ul{u})_{i\bi} - \fg_{i\bi})
%\Big(1 + \sum (\fg^{i\bi})^2\Big)
\end{aligned}
\end{equation}

Let $\phi = e^{A \eta}$ where $\eta = \ul{u} - u + \sup_M (u - \ul{u})$ and
$A$ is a positive constant. We see that $\phi_i = A \phi \eta_i$ and
$\phi_{i\bi} =  A \phi \eta_{i\bi} + A^2 \phi \eta_i \eta_{\bi}$.
By Schwarz inequality,
\begin{equation}
\label{gblq-C200a}
  \begin{aligned}
%2 w \fg^{i\bi} \fRe\{\phi_i \ol{\lambda_i}\}
     %= \,& 2 w A \phi \fg^{i\bi} \fRe\{\eta_i \ol{\lambda_i}\} \\
 2 A F^{i\bi} \fRe\{\eta_i \ol{\lambda_i}\}
 \geq \,& - A^2 F^{i\bi} |\eta_i|^2 - C. % \phi \sum F^{i\bi}.
     %     - \frac{C \phi}{w} \sum \frac{1}{F^{i\bi}} \\
 %\geq \,& - w A^2 \phi F^{i\bi} |\eta_i|^2 - C w \phi \Big(1 + \sum F^{i\bi}\Big)
\end{aligned}
\end{equation}
By Lemma~\ref{gblq-lemma-C20}  we see that
\begin{equation}
\label{gblq-C200}
  \begin{aligned}
\phi^{-1} F^{i\bi} \phi_{i\bi}
   = \,& A F^{i\bi} \eta_{i\bi} + A^2 F^{i\bi} \eta_i \eta_{\bi} \\
    = \,& A F^{i\bi} (\chi_{i\bi} + \ul{u}_{i\bi}) - \frac{n A}{\psi}
     + A^2 F^{i\bi} \eta_i \eta_{\bi} \\
\geq \,& \frac{ \theta A}{\psi} + A^2 F^{i\bi} \eta_i \eta_{\bi}
 \end{aligned}
 \end{equation}
provided $W$ is sufficiently large. This
combined with \eqref{gblq-C200a} and \eqref{gblq-C94} gives
\[ (\theta A - C \psi) W \leq C. \]
Choosing $A$ large enough we derive a bound $W \leq C$.
\end{proof}

By Proposition~\ref{gblq-prop-C10} equation~\eqref{CH-I10} is uniformly
elliptic for solutions $u$ with $\chi_u > 0$.
Therefore one can apply Evans-Krylov Theorem and the Schauder theory to
derive higher order estimates.

\bigskip

%\small

\medskip
%Received March 2012; revised March 2012.
\medskip

\end{document}